\documentclass[epic,eepic,11pt]{amsart}

\usepackage{amsfonts,mathrsfs}
\usepackage[mathscr]{eucal}
\usepackage{mathtools}
\usepackage{amssymb}
\usepackage{amscd}
\usepackage{fancyhdr}
\usepackage[all,cmtip]{xy}

\usepackage{euler,eucal}

\DeclareMathAlphabet{\mathbf}{T1}{ppl}{bx}{n}
\DeclareMathAlphabet{\mathrm}{T1}{ppl}{m}{n}

\numberwithin{equation}{section}

\newcommand\note[1]%
{$^\dagger$\marginpar{\footnotesize{$^\dagger${#1}}}}

\def\({\left(}
\def\){\right)}
\def\<{\left<}
\def\>{\right>}

\newtheorem{theorem}{Theorem}[section]
\newtheorem{proposition}[theorem]{Proposition}
\newtheorem{lemma}[theorem]{Lemma}
\newtheorem{definition}[theorem]{Definition}

\theoremstyle{definition}

\newtheorem{remark}[theorem]{Remark}

\newcommand\bb[1]{{\text{\bf#1}}}

\newcommand\Z{\bb{Z}}

\newcommand\R{\mathbb{R}}

\newcommand     {\comment}[1]   {}
\newcommand{\mute}[2] {}
\newcommand     {\printname}[1] {}

\newcommand\funclim[1]{\operatorname*{\mathrm{#1}}}

\renewcommand\lim{\funclim{lim}}

\newcommand\sur{\mathrel{\to\kern-1.8ex\to}}
\newcommand\iso{\mathrel{\hookrightarrow\kern-1.8ex\to}}

\newcommand\longhookrightarrow{\lhook\joinrel\longrightarrow}

\newcommand\longsur{\mathrel{\longrightarrow\kern-1.8ex\to}}
\newcommand\longiso{\mathrel{\longhookrightarrow\kern-1.8ex\to}}

\begin{document}

\bibliographystyle{amsalpha}
\date{\today}

\title{Symplectic Harmonic theory and the Federer-Fleming deformation theorem}

\author{Yi Lin }

 \dedicatory{Dedicated to Late Professor H. Federer with admiration.}
 \date{\today}

\maketitle
\begin{abstract} In this article, we initiate a geometric measure theoretic approach to symplectic Hodge theory. In particular, we apply one of the central results in geometric measure theory, the Federer-Fleming deformation theorem,  together with the cohomology theory of
normal currents on a differential manifold, to establish a fundamental property on symplectic Harmonic forms. We show that on a closed symplectic manifold, every real primitive cohomology class of positive degrees admits a symplectic Harmonic representative not supported on the entire manifold. As an application, we use it to investigate the support of symplectic Harmonic representatives of Thom classes, and give a complete solution to an open question asked by Guillemin.


\end{abstract}



\setcounter{section}{0} \setcounter{subsection}{0}

\section{Introduction}

Symplectic Hodge theory was introduced by Ehresmann and Libermann \cite{EL49}, \cite{L55}, and was rediscovered by Brylinski \cite{brylinski;differential-poisson}. By mimicking the construction in Riemannian Hodge theory, one can define the symplectic Hodge star operator $\star$. In this context, a differential form $\alpha$ is said to be symplectic Harmonic if and only if $d\alpha=d^{\Lambda} \alpha=0$, where $d^{\Lambda}=\pm \star d \star$ is the symplectic Hodge adjoint operator.


In contrast with the usual Riemannian Hodge theory, symplectic Hodge theory is not associated with elliptic operators. As a result, Harmonicity is a much flabbier property in symplectic Harmonic theory. One does not expect the Harmonic representative of a Thom class to exhibit any interesting global features. However, Bahramgiri showed in his MIT thesis \cite{Ba06} that this is not the case.

Indeed, Bahramgiri proved \cite{Ba06} that the Thom class of a compact oriented coisotropic submanifold of a symplectic manifold always admits a Harmonic representative not supported on the entire manifold. This stands in contrast with Riemannian Hodge theory, where any Harmonic form with a zero of infinite order is identically zero, cf. \cite{AKS62}. In addition, Bahramgiri also proved that any symplectic Harmonic representative of the Thom class of a symplectic submanifold is nowhere vanishing. This motivated Victor Guillemin to ask the following fundamental question in symplectic Hodge theory.\\

\textbf{ Question}:  What can we say about the support of symplectic Harmonic representatives of the Thom classes of isotropic submanifolds? More generally, can we give a characterization of the submanifolds of a symplectic manifold whose Thom class admits a symplectic Harmonic representative that is not supported on the entire manifold?\\

In a different direction, motivated by Chern Bott cohomology in complex geometry and by string theory, L. Tseng and S. T. Yau studied cohomology theories on symplectic manifolds in recent works \cite{TY09}, \cite{TY10}, \cite{TY11}. They discovered a remarkable fact that on a symplectic manifold there is an elliptic complex on the space of primitive differential forms, which is defined using only the symplectic structure. They went on to develop primitive cohomology theories naturally associated with this elliptic complex, and showed by examples that these cohomologies naturally gave rise to new symplectic invariants especially for non-K\"ahler symplectic manifolds.

In addition, Tseng and Yau \cite{TY09} proposed a definition of primitive homology using coisotropic chains, and proved that there is a natural homomorphism from the primitive homology to the primitive cohomology. In view of this result,  one naturally wonders whether or not every primitive cohomology class is represented by a coisotropic cycle (defined in an appropriate sense).

On a $2n$ dimensional symplectic manifold $(M,\omega)$, a cohomology class $\gamma \in H^p(M,\R)$
is said to be primitive if and only if $[\omega^{n-p+1}]\cup \gamma=0$, where $0\leq p\leq n$. D. Yan \cite{Yan96} had a simple algebraic proof that a primitive cohomology class always admits a symplectic Harmonic representative, even if the symplectic manifold does not satisfy the Hard Lefschetz property.

The Thom classes of compact oriented coisotropic submanifolds provide important examples of primitive cohomology classes of positive degrees. Moreover, Bahramgiri \cite{Ba06} proved that these classes are represented by Harmonic forms not supported on the entire manifold. It naturally leads to the following question concerning symplectic Harmonic forms. \vskip 2mm
 {\bf Question 1}:  On a closed symplectic manifold, is every primitive cohomology class of positive degrees represented by a symplectic Harmonic form not supported on the entire manifold? \vskip 2mm

 It is well known that one can smoothfy a closed De Rham current to get a closed differential form in the same cohomology class. On a symplectic manifold, using Bahramgiri's construction of symplectic smoothing operator \cite{Ba06}, one can smoothfy a closed primitive current and get a closed primitive differential form in the same cohomology class. Since a closed primitive differential form is automatically symplectic Harmonic, Question 1 has the following equivalent form.
 \vskip 2mm
  {\bf Question 2}: On a closed symplectic manifold, is every primitive cohomology class of positive degrees represented by a closed primitive De Rham current not supported on the entire manifold?
  \vskip 2mm

Indeed, Tseng and Yau \cite{TY10} showed that the canonical current of a compact oriented submanifold is primitive if and only if the submanifold is coisotropic. Clearly, canonical currents of submanifolds can never be supported on the entire manifold. In view of these observations, a positive answer to Question 2 indicates that every primitive cohomology class of positive degrees is represented by a coisotropic cycle in some weak sense.

In the literature, similar questions have been considered by R. Schoen and J. Wolfson in the context of Lagrangian homology classes. In a series of highly influential papers \cite{SW99}, \cite{SW01},  Schoen and Wolfson studied the Lagrangian plateau problem in the parametric setting.  As a preliminary consideration, they characterized the Lagrangian homology classes of a symplectic manifold.

On a $2n$ dimensional symplectic manifold $(M,\omega)$, an $n$ dimensional integral homology class $[\alpha]\in H_n(M,\Z)$ is called Lagrangian if it can be represented by a Lagrangian cycle. More precisely, an integral cycle is called Lagrangian if it is represented by n-simplicies that are the images of piecewise $C^1$ Lagrangian maps. It was shown for $n=2$ in \cite{SW01}, for $n=3$  in \cite{W00} when $M$ is simply connected, and for arbitrary $n$ in \cite{W04} when $[\omega]$ is integral, that an integral homology class $[\alpha]\in H_n(M,\Z)$ is Lagrangian if and only if $[\omega]\cap [\alpha]=0$. Equivalently, an integral homology class is Lagrangian if and only if its Poincar\'e dual is a primitive cohomology class.

It is important to note that the existing works on Lagrangian homology classes all depend on the H principles \cite{Gro86} for Lagrangian or contact immersions. For instance, Wolfson \cite{W04} uses successively a horizontal extension lemma due to Gromov \cite[3.5]{Gro96}. However, for simple dimensional reasons, these $H$-principle related techniques do not seem to have an extension to the coisotropic setting.

Inspired by the above-mentioned pioneering work of Tseng and Yau, and by Guillemin's question concerning symplectic Harmonic representatives of Thom classes,  we develop in the present paper a new approach to symplectic Hodge theory via geometric measure theory.
In particular, we apply one of the central analytic tools in geometric measure theory, the Federer-Fleming deformation theorem, together with the cohomology theory of normal currents on a differential manifold, to answer Question 2 in the affirmative. This immediately implies the following result.
\begin{theorem}\label{main-result1} Let $M$ be a closed symplectic manifold. Then every primitive cohomology class of positive degrees on $M$ admits a symplectic Harmonic representative not supported on the entire manifold.
\end{theorem}
As an immediate application of Theorem \ref{main-result1}, we provide a complete solution to Guillemin's question concerning symplectic Harmonic representatives of Thom classes by establishing the following result.

 \begin{theorem} \label{Guillemin's-question} Assume that $(M,\omega)$ is a $2n$ dimensional connected compact symplectic manifold, and that $N$ is a compact oriented submanifold of $M$ whose Thom class $[\tau_N]$ admits a symplectic Harmonic representative. If $\text{\,codim\,}(N)$  is odd,
 then $[\tau_N]$ must always admit a symplectic Harmonic representative not supported on the entire manifold. If $\text{\,codim\,}(N)=2p$ is even, then a sufficient and necessary condition for $[\tau_N]$ to admit a symplectic Harmonic representative not supported on the whole manifold is that $[\omega]^{n-p}\wedge [\tau_N]=0$. As a special case, the Thom class of a compact oriented isotropic submanifold always admits a symplectic Harmonic representative not supported everywhere on $M$.
 \end{theorem}


 In the past, symplectic Harmonic theory has been studied mostly using algebraic tools such as Lie algebra representation and spectral sequences.
 The geometric measure theoretic approach to symplectic Hodge theory developed in the current paper is original. It allows us to establish deeper properties in this theory which can not be accessed using the usual algebraic methods.


The plan of this paper is as follows. Section \ref{preliminary} presents a quick review of symplectic Hodge theory, and the theory of compactly supported currents on a manifold. Section \ref{normal-currents} gives a self-contained intrinsic description of the cohomology theory of normal currents on a differential manifold, as well as the precise statement of the Federer-Fleming Deformation theorem.  Section \ref{sl2-on-distribution-de-rham} discusses the canonical $sl_2$ module structure on the space of compactly supported De Rham currents on a symplectic manifold. Section \ref{fundamental-result} proves that on a compact symplectic manifold every primitive cohomology class of positive degrees admits a
symplectic Harmonic representative not supported on the entire manifold. Finally, Section \ref{open-question} provides a complete solution to the open question asked by Guillemin concerning the symplectic Harmonic representatives of Thom classes.

\section{Preliminaries}\label{preliminary}
\subsection{Review of Symplectic Hodge theory}\medskip \noindent\vskip 0.5cm

In this section we present a brief review of background materials in symplectic Hodge theory. For more details, we refer to \cite{brylinski;differential-poisson}, \cite{Yan96}, \cite{Ba06}, \cite{TY09} and \cite{TY10}. Throughout this section, we assume that
$(M,\omega)$ is a $2n$ dimensional symplectic manifold.

On the symplectic manifold $(M,\omega)$, the Lefschetz map $L$, the dual Lefschetz map $\Lambda$, and the degree counting map $H$ are defined as follows.

\begin{equation} \label{three-canonical-maps} \begin{aligned} & L :\Omega^*(M) \rightarrow \Omega^{*+2}(M), \,\,\,\alpha \mapsto \alpha \wedge \omega,\\
 & \Lambda: \Omega^*(M) \rightarrow \Omega^{*-2}(M),\,\,\,\alpha \mapsto \iota_{\pi}\alpha,\\
& H: \Omega(M)\rightarrow \Omega(M),\,\,\,H(\alpha)=\displaystyle\sum_{k=0}^n (n-k)\Pi^k(\alpha),\end{aligned}\end{equation}
where $\pi=\omega^{-1}$ is the canonical Poisson bi-vector associated to $\omega$, and \[\Pi^k:\Omega(M)=\bigoplus_{i=0}^{2n}\Omega^i(M) \rightarrow \Omega^k(M)\] is the projection map.

The actions of $L$, $\Lambda$ and $H$ on $\Omega(M)$ satisfy the following commutator relations.

\begin{equation} \label{sl2-module-on-forms} [  \Lambda, L]=H, \,\,\,[H, \Lambda]=2\Lambda,\,\,\,[H, L]=-2L.
\end{equation}

Therefore they define a representation of the Lie algebra $sl(2)$ on $\Omega(M)$. Although the $sl_2$-module $\Omega(M)$ is infinite dimensional, there are only finitely many eigenvalues of the operator $H$. $sl_2$-modules of this type are studied in great details in \cite{Ma95} and \cite{Yan96}.  Among other things,the following result is proved in \cite{Yan96}.

\begin{lemma} \label{yan's-result}  Assume that $(M,\omega)$ is a $2n$ dimensional symplectic manifold.
\begin{enumerate}
\item [1)] For any $0\leq r \leq n$, the Lefschetz map
\[ L^{n-r}: \Omega^{r}(M) \rightarrow \Omega^{2n-r}(M), \,\alpha \mapsto \omega^{n-r}\wedge \alpha\]
is an isomorphism.
\item [2)] Let $\alpha\in \Omega^k(M)$ with $0\leq k\leq n$. Then $\alpha$ is primitive, i.e., $\omega^{n-k+1}\wedge \alpha=0$,  if and only if $
\Lambda \alpha=0$.
\item [3)]  Any differential form
$\alpha_k \in \Omega^k(M)$ admits a unique Lefschetz decomposition
\begin{equation}\label{lefschetz-decompose-forms} \alpha_k =\displaystyle \sum_{r\geq \text{max}(\frac{k-n}{2}, 0)} \dfrac{L^r}{r!}\beta_{k-2r},\end{equation}
where $\beta_{k-2r}$ is a primitive form of degree $k-2r$.

\end{enumerate}
\end{lemma}

Since the symplectic structure $\omega$ is a non-degenerate two form, using it to identify one forms with one vectors we obtain
a non-degenerate bi-linear pairing on the space of one forms. This pairing further extends to a non-degenerate bi-linear pairing
$(\cdot,\cdot)$ on the space of differential $k$-forms. In this context, we define the symplectic Hodge star
operator $\star$ as follows.
\begin{equation}\label{sym-hodge-star}  \star\alpha_k\wedge \beta_k=(\alpha_k,\beta_k)\dfrac{\omega^n}{n!},
\end{equation}
where both $\alpha_k$ and $\beta_k$ are differential $k$-forms.

On the space of differential $k$-forms,  the symplectic Hodge adjoint operator of the exterior differential $d$,
is given by \[ d^{\Lambda}\alpha_k=(-1)^{k+1}\star d\star\alpha_k.\]

It is straightforward to check that $d$ anti-commutes with $d^{\Lambda}$. In this context, a differential form $\alpha$ is said to be symplectic Harmonic if and only if $d\alpha=d^{\Lambda}\alpha=0$.

The following commutator relations are important.
\begin{equation} \label{commutator-on-forms}\begin{matrix}
 &[d,L]=0,  &[d^{\Lambda},\Lambda]=0,  &[d,\Lambda]=d^{\Lambda},
 \\&[d^{\Lambda}, L]=d,  &[dd^{\Lambda},L]=0,  &[dd^{\Lambda},\Lambda]=0.\end{matrix}
\end{equation}

We will need the following refinement of Lemma \ref{yan's-result}.
\begin{lemma}\label{yan's-result-2} Consider the Lefschetz decomposition of the differential form $\alpha_k$ as given in Equation (\ref{lefschetz-decompose-forms}). Then there are non-commutative polynomials $\Phi_{k,r}(L,\Lambda)$ such that
\[ \beta_{k-2r}=\Phi_{k,r}(L,\Lambda)\alpha_k;\] moreover,
each $\beta_{k-2r}$ is $d$-closed and primitive if $\alpha_k$ is Harmonic.

\end{lemma}

\begin{proof} For the first assertion, we refer to \cite[Thm. 3.12]{We80} for a detailed proof. The second assertion in Lemma \ref{yan's-result-2}
follows directly from the commutator relations given in Equation (\ref{commutator-on-forms}).

\end{proof}

It is well known that on a differential manifold one can apply a smoothing operator to a current and get a smooth differential form.
 In the proof of \cite[Theorem 2]{Ba06}, Bahramgiri constructed a symplectic smoothing operator for currents on symplectic manifolds.
This is an important construction in symplectic Harmonic theory. We summarize the properties of the symplectic smoothing operator in the following theorem.

\begin{theorem}\label{symplectic-smoothing-operator}(\cite{Ba06}) Let $(M, \omega)$ be a symplectic manifold, and $T$ a current of degree $k$ supported inside an open set $W$ in $M$. Then there exists a symplectic smoothing operator $\mathcal{S}$ such that
 $\mathcal{S}(T)$ is a smooth differential $k$-form  supported inside $W$. Moreover, if $T$ is a closed current,
then $S(T)$ is a closed differential form such that $S(T)-T$ is a coboundary (in the space of De Rham currents);
if  $T$ is symplectic Harmonic, then $\mathcal{S}(T)$ is a symplectic Harmonic differential form;
and if $T$ is primitive, then $\mathcal{S}(T)$ is a primitive differential form.
\end{theorem}

\subsection{Review of the theory of compactly supported currents}\label{review-of-currents}

\medskip \noindent

\vskip .5cm

In this section, we present a quick review of the standard theory of compactly supported currents in the context of locally convex topological vector spaces \cite{R73}. We follow closely the exposition given in the classic textbook \cite{DeRham84}. Throughout this section, we assume that $M$ is an $m$ dimensional differential manifold. We denote by $\mathcal{E}(M)$, or simply $\mathcal{E}$, the space of $C^{\infty}$
forms on $M$. Moreover, we denote by $\mathcal{E}^k(M)$, or simply $\mathcal{E}^k$, the space of differential $k$-forms.

  For any non-negative integer $i\geq 0$,  and any compact set $K$ which lies in an open coordinate neighborhood $U$,
  we define a semi-norm $\vert\vert \cdot\vert\vert_{K}^i$ on $\mathcal{E}$ as follows: using the coordinate system
  $\{ x_1, x_2, \cdots, x_m\}$ on $U$, if the restriction of a $k$-form $\phi$ to $U$ has an expression as
\[\displaystyle\sum_{i_1<i_2<\cdots<i_k} f_{i_1i_2\cdots i_k}dx_{i_1}\wedge d_{i_2}\wedge \cdots \wedge dx_{i_k} ,\]
 then
\begin{equation}\label{semi-norm} \vert\vert \phi\vert\vert_{K}^i= \text{sup} \{ \vert D^{j} f_{i_1i_2\cdots i_k}\vert,\, 0\leq \vert j\vert \leq i, x\in K\},\end{equation}
where $j=(j_1, j_2, \cdots, j_m)$ is a multi-index, $\vert j\vert =j_1+j_2+\cdots +j_m$, $D^j=D_1^{j_1}\cdots D_m^{j_m}$,
and $D_l=\dfrac{\partial}{\partial x_l}$.
The family of all such semi-norms induces a translation invariant Hausdorff topology on $\mathcal{E}$, and turns $\mathcal{E}$ into a locally convex topological vector space.



   We say that a continuous functional $T \in \mathcal{E'}$ is zero on an open set $V\subset M$ if $T(\phi)=0$ for any $\phi\in\mathcal{E}$ supported inside $V$. It is an elementary fact that there exists a maximal open set in $M$ such that $T$ is zero. The complement of this maximal open set is called the support of $T$, and is denoted by $\text{supp}\,T$. Moreover, it is easy to show that for any $T\in \mathcal{E}'$, $\text{supp}\,T$ is a compact subset of $M$, c.f. \cite[Sec.10]{DeRham84}.

   \begin{definition} Any continuous functional in $\mathcal{E}'(M)$ is called a compactly supported current on $M$.
   \end{definition}

 We say that a compactly supported current $T$ in $\mathcal{E}'$ is $k$ dimensional if $T(\phi)=0$ for any
$\phi \in \mathcal{E}(M)$ whose degree is not $k$. For a $k$ dimensional compactly supported current $T$, its degree is defined to be $m-k$.
We denote by $\mathcal{E'}^k$ the space of all compactly supported currents in $\mathcal{E}'$ of degree $k$.

  Given a $q$ dimensional current $T$ in $\mathcal{E}'$, its boundary  is by definition a $q-1$ dimensional current given by

\begin{equation}  \label{boundary-operator} \partial T(\phi)= T(d\phi), \,\,\,\,\,\,\,\,\forall\,\phi \in \mathcal{E}^{q-1},\end{equation}
and its exterior differential is defined by
\begin{equation}  \label{differential-of-current} dT=(-1)^{m-q+1} \partial T .  \end{equation}

Note that the exterior differential $d: \mathcal{E}\rightarrow \mathcal{E}$ is a continuous linear mapping. Therefore if $T\in \mathcal{E}'$, then both $\partial T$ and $dT$ are well defined compactly supported currents in $\mathcal{E}'$.
Since $d ^2=\partial^2=0$, we have a differential complex
\begin{equation}\label{distributional-derham-complex}0\rightarrow \mathcal{E'}^0\xrightarrow{d}\mathcal{E'}^1\xrightarrow{d}\cdots \xrightarrow{d}\mathcal{E'}^{m}\rightarrow 0.\end{equation}
The $i$-th compactly supported distributional De Rham cohomology $H_c^{i,-\infty}(M)$ is defined to be the $i$-th cohomology of the differential complex (\ref{distributional-derham-complex}).


Let $f$ be a smooth map from a manifold $X$ into another manifold $Y$. Then the pullback map $f^*: \mathcal{E}(Y) \rightarrow \mathcal{E}(X)$ is a continuous linear mapping. Thus it induces a pushforward map from $\mathcal{E}'(X)$ to $\mathcal{E}'(Y)$ in the following way.
\[f_*: \mathcal{E}'(X)\rightarrow \mathcal{E}'(Y),\,\,\, f_*(T)(\phi)=T(f^*\phi), \,\forall\, \phi\in \mathcal{E}(Y).\]
It is straightforward to check that this pushforward map is a continuous map; moreover, it  commutes with the boundary map, and sign commutes with the differential $d$. More precisely, let $\text{dim}\, X= m$ and $\text{dim}\, (Y)=n$, then we have that
\[ \partial f_*(T)=f_*(\partial T),\,\,\,\,\, df_*(T)=(-1)^{m+n}f_*dT,\,\,\,\forall\, T\in \mathcal{E}'(X).\]

We are particularly interested in the following situation. Let $U$ be an open set of a manifold $M$. Then the inclusion map
$i: U\hookrightarrow M$ induces a pushforward map
\begin{equation} \label{topological-embedding} i_*: \mathcal{E}'(U)\rightarrow \mathcal{E}'(M).\end{equation}
It is easy to show that Map (\ref{topological-embedding}) is a homeomorphism onto its image. So we would not distinguish a
current $T\in \mathcal{E}'(U)$ from its image $i_*(T) \in \mathcal{E}'(M)$ under the pushforward map.


Now suppose that $M$ is an $m$ dimensional manifold covered by two open sets $U$ and $V$.  Then
 sequences of inclusions
\[ M\longleftarrow U \sqcup V \longleftarrow U\cap V\]
give rise to sequences of pushforward maps
\begin{equation*}
\begin{aligned}
     \mathcal{E'}^i(M)\xleftarrow{\text{sum}} \mathcal{E'}^i(U)\oplus \mathcal{E'}^i(V)
                      \xleftarrow{\text{signed pushforward}} & \mathcal{E'}^i(U\cap V)\\
    (-i_*(T), i_*(T)) \xleftarrow[\rule{2.39cm}{0cm}]{} & T
\end{aligned}
\end{equation*}
This gives us a Mayer-Vietoris sequence
\begin{equation}\label{Mayer-Vietois-sq-01}
0\longleftarrow \mathcal{E'}^i(M)\longleftarrow \mathcal{E'}^i(U)\oplus \mathcal{E'}^i(V) \longleftarrow \mathcal{E'}^i(U\cap V)\longleftarrow 0.
\end{equation}

\begin{proposition}\label{mayer-vietoris}
The Mayer-Vietoris sequence (\ref{Mayer-Vietois-sq-01}) of compactly supported currents is exact.

\end{proposition}

\begin{proof} The argument given in \cite[Prop. 2.7]{BT82} for the exactness of the Mayer-Vietoris sequence of forms with compactly support extends
to the present situation.
\end{proof}

Therefore we have a long exact sequence

\begin{equation}\label{exact-sq-1} \cdots \leftarrow H_c^{-\infty,i}(U\cup V)\leftarrow  H_c^{-\infty,i}(U)\oplus H_c^{-\infty,i}(V)\leftarrow H_c^{-\infty,i}(U\cap V) \leftarrow H_c^{-\infty,i-1}(U\cup V)\leftarrow \cdots
\end{equation}

Now we assume that $M$ is an oriented compact manifold. Since $M$ is oriented and compact, any differential form $\alpha \in \mathcal{E}(M)$ defines a compactly supported current in $\mathcal{E}'(M)$ as follows.
\[ \mathcal{E}(M)\rightarrow \R,\,\,\, \,\,\,\phi\mapsto \int_M\alpha \wedge \phi.   \]

Therefore there is a natural chain homomorphism
\begin{equation}\label{chain-map-1} \mathcal{E}(M) \rightarrow \mathcal{E'}(M).\end{equation}

The following  result are standard and we refer to \cite{DeRham84} for detailed proofs.

\begin{theorem}\label{natural-maps-forms-and-currents} The homomorphism (\ref{chain-map-1}) induces an isomorphism.
\[ H_{\text{DR}}^*(M) \cong  H_c^{-\infty,*}(M),\]
where $H_{\text{DR}}^*(M)$ denotes the usual De Rham cohomology of the manifold $M$.
\end{theorem}




\section{Normal currents and the deformation theorem }\label{normal-currents}
\medskip \noindent
 \vskip 0.5cm
In the existing literature of geometric measure theory, the homology (or cohomology) theory of normal currents on a differential manifold $M$ is usually treated under the assumption that $M$ is a submanifold of some Euclidean space $\R^N$, c.f. \cite[ch. 5]{GSM98}. In this treatment, a compactly supported current $T$ is said to be normal in $M$ if it is a normal current in $\R^N$ in the sense of \cite[4.1.7]{F69}, and if its support is contained in $M$. This approach is sound due to Federer's flatness theorem, and leads directly to the homology theory of the manifold $M$.

However, to work with normal currents on a symplectic manifold, or on a manifold with some extra geometric structures, it will be more convenient to have an intrinsic description of normal currents on a differential manifold.  For this reason, we present in this section an elementary self-contained account of normal currents on a manifold without assuming that the manifold is embedded into an ambient Euclidean space. We then use normal currents to develop the cohomology theory
of a differential manifold. Since the Federer-Fleming deformation theorem plays a central role in our approach to symplectic Harmonic theory, we also include in this section a precise statement of the deformation theorem that we need.




 Let $(\cdot,\cdot)$ be an inner product on $\R^m$. For any $0\leq k\leq m$, the inner product induces a norm on $\wedge_k\R^m$ which we denote by $\vert\cdot \vert$. A $k$-vector
$w\in \wedge_k\R^m$ is called simple if
\[ w=v_1\wedge\cdots \wedge v_k\]
for some collection of vectors $\{v_1,\cdots,v_k\}\subset \R^m$. The {\bf comass} of a $k$-form $\varphi\in \wedge^k \R^m$ is defined as
\begin{equation}\label{comass}  \vert\vert \varphi\vert \vert:= \text{\,sup}\{ <\varphi, w>,\,\,w\in\wedge_k\R^m\,\text{is simple and}\, \vert w\vert\leq 1\}.
\end{equation}

Now let $M$ be a differential manifold. Choose a Riemannian metric $g$ on $M$. Then $\forall\, x\in M$,
the Riemannian metric induces an inner product on the tangent space $T_x M$. Thus for any differential form $\varphi\in \mathcal{E}(M)$,
 and for any $x\in M$, we have a pointwise co-mass $\vert\vert \varphi(x)\vert\vert$ which is defined as we described in Equation (\ref{comass}).

For a current $T \in \mathcal{E'}(M)$, its {\bf mass norm} is defined to be
\[ \vert\vert T\vert\vert=\text{sup}\,\{<T,\varphi>, \,\,\varphi\in \mathcal{E}(M),\, \vert\vert\varphi(x) \vert \vert\leq 1,\,\forall\, x\in M\}.\]


Let $K\subset M$ be a compact set, and let $\mathcal{E}_K'(M)$ be the set of De Rham currents on $M$ supported inside $K$. Then it is easy to see that different choices of Riemannian metrics on $M$ would yield equivalent mass norms on $\mathcal{E'}_K(M)$. As a result, if a compactly supported current $T\in\mathcal{E}'(M)$ has finite mass with respect to a given Riemmanian metric on $M$, it has finite mass with respect to any Riemannian metric on $M$. Thus the notion of a compactly supported current with {\bf finite mass} on $M$ is well defined without reference to any particular Riemannian metric.

\begin{definition}\label{normal-current} Let $V$ be an open subset of a differential manifold $M$ with a given Riemannian metric $g$.
A current $T$ is called a normal current in $V$ if it is compactly supported inside $V$ such that
\[{\bf N}(T):= \vert\vert T\vert\vert+\vert\vert\partial T\vert\vert<\infty.\]
For any compact subset $K$ in $U$, we define $N_K(V)$ to be the space of normal currents supported inside $K$, and define $N(V)$ to be
the vector space of normal currents supported inside $V$. We denote by $N^i(V)$ the space of normal currents supported inside $V$ which
are of degree $i$.

\end{definition}

\begin{remark} By the discussion in the paragraph preceding Definition \ref{normal-current}, the notion of normal currents on a differential manifold is independent of the choices of Riemannian metrics. \end{remark}


Note that if $M=U$ is an open subset of $\R^m$, and if the Riemannian metric $g$ is given by the standard Riemannian metric on $\R^m$,
the notion of normal currents given in Definition \ref{normal-current} is exactly the one given in \cite[4.1.7]{F69}. However, since
different choices of Riemannian metric on $U$ would result in equivalent mass norms on $\mathcal{E'}_K(U)$ for any compact subset $K\subset U$, our definition of normal currents is equivalent to the one given in \cite[4.1.7]{F69} even if $g$ is different from the standard metric on $\R^m$.  These observations lead to the following simple lemma.

\begin{lemma}\label{compatibility1} Let $M$ be an $m$ dimensional differential manifold, let $U$ and $V$ be open subsets of $\R^m$ and $M$ respectively, and let $ \varphi: U\rightarrow V$ be a diffeomorphism. Then the diffeomorphism $\varphi$ induces an isomorphism of the space of normal currents
\[ \varphi_{*}: N(U)\rightarrow N(V),\] where $N(U)$ is given in the sense of \cite[4.1.7]{F69}, and $N(V)$ is given as in Definition \ref{normal-current}.
\end{lemma}

In an Euclidean space, polyhedral chains form an important subclass of normal currents. Since any differential manifold admits a polyhedral subdivision, the notion of polyhedral chain can easily be extended to a differential manifold. Indeed, polyhedral chains associated to a polyhedral subdivision has already been used extensively in \cite{DeRham84} to establish the Poincar\'e duality of a differential manifold.  Following closely the
 expositions given in \cite[Sec. 21]{DeRham84}, we give a quick
review of some basic facts on polyhedral chains that we need later in this paper.

 A $p$-simplex $\sigma$ in $\R^{m}$ is the convex hull of $p+1$ geometrically independent points
$\{v_0, v_1, \cdots, v_p\}$, i.e.,
\[ \sigma=\{x_0v_0+x_1v_1+\cdots + x_p v_p\in \R^{m},\,\,\, x_0+x_1+\cdots + x_p =1,\, x_i\geq 0,\, i=0,1,\cdots,p\}.\]
The $p$-simplex $\sigma$ is called an oriented $p$-simplex, and is denoted by $[v_0,v_1,\cdots,v_p]$, if it is equipped with an orientation induced by the multi-vector
\[(v_1-v_0)\wedge (v_2-v_0)\wedge\cdots \wedge (v_p-v_0).\]


\begin{definition}\label{polyhedral-chain-in-Rm} A compactly supported current $T$ in $\mathcal{E}'(\R^m)$ is called a $k$ dimensional real polyhedral chain in
$\R^m$ if there exist
real scalars $a_1,\cdots, a_l$ and oriented $k$-simplices $\sigma_1,\cdots, \sigma_l$ such that
\begin{equation}\label{non-over-lapping}  T(\phi)=\displaystyle \sum_{i=1}^l a_i\int_{\sigma_i}\phi,\,\,\,\forall\, \phi\in \mathcal{E}(\R^m).\end{equation}
\end{definition}

Let $M$ be an $m$ dimensional differential manifold. A set $\tau \subset M$ is called a $p$ dimensional {\bf cell}, if there is a $p$-simplex $\sigma$ in a subspace $\R^p$ of $\R^m$, and a diffeomorphism $\Phi$ of an open neighborhood of
$\sigma$ in $\R^m$ onto an open neighborhood of $\tau$ in $M$ which maps $\sigma$ onto $\tau$. The image by $\Phi$ of the interior points and frontier
points of $\sigma$ are called respectively interior points and frontier
points of the cell. Moreover, $\tau$ is called an oriented $p$ dimensional cell if the
$p$-simplex $\sigma$ is oriented. Clearly, each oriented cell defines a compactly supported current on $M$ in a canonical way.
From now on, we will not distinguish an oriented cell from the canonical current that it induces.

A {\bf polyhedra subdivision} of $M$ consists of a set $S$ of cells of $M$ that satisfies the following conditions:
 \begin{itemize} \item [1)] $S$ is locally finite, that is, every compact subset of $M$ meets only finitely many
cells of $S$;
\item [2)] Each point of $M$ is contained in the interior point of exactly one cell of $S$;

\item [3)] The set of frontier points of each cell of $S$ is the union of cells of $S$.

\end{itemize}

Fix a polyhedral subdivision $S$ on the differential manifold $M$. By abuse of language, we call an oriented cell $\tau$ of $M$ an oriented cell of $S$, if as a set, $\tau$ is a cell of $S$.  We say that $T$ is a $p$ dimensional polyhedral chain of the subdivision \footnote{It is called an odd chain of the subdivision in \cite[Sec.21]{DeRham84}. Since we are only going to work with oriented manifolds in this paper, we do not have to distinguish odd chains and even chains.}  if there exist oriented $p$ dimensional cells $\tau_1,\cdots, \tau_n$ of $S$, and scalars $a_1,a_2,\cdots, a_n$,
such that \[ T=\displaystyle \sum_{i=1}^na_i\tau_i.\]
The following result is an immediate consequence of \cite[Thm. 16]{DeRham84}.

\begin{theorem}\label{duality} Let $M$ be an oriented differential manifold with a given polyhedral subdivision $S$. Then for any compactly supported
closed differential form $\alpha$, there exists a polyhedral chain of the subdivision, say $T$, and a compactly supported current $\Gamma$, such that
$\alpha-T=d\Gamma$.
\end{theorem}

It is important to note the invariance of normal currents under Lipschitzian maps. Let $U$ be an open set in $\R^n$, and $X\subset U$. A map $f: X\rightarrow \R^m$ is said to be an $L$-lipschitzian map from $X$ to $\R^m$ if there exists a constant $L\geq 0$ such that \[ \vert f(x)-f(y)\vert \leq L\vert x-y \vert,\,\,\,\,\forall\, x,y\in X.\] Every Lipschitzian map has a least Lipschitz constant, which is denoted by $Lip(f)$. A map $f:U\rightarrow \R^m$ is said to be locally Lipschitzian if for any compact subset $K\subset U$, $f\vert_K$ is a Lipschitzian map. In particular, any smooth map from $U$ to $\R^m$ is a locally Lipschitzian map. We refer to \cite[2.10.43, 4.1.14]{F69} for a proof of the following simple result.

\begin{lemma}\label{Lipschitz-invariance-1} Let $U$ and $V$ be open subsets of $\R^n$ and $\R^m$ respectively, let $K$ be a compact subset of $U$, and let $f:U\rightarrow V$ be a locally Lipschitzian map. Then for any $T\in N_{p,K}(U)$, we have that
\[ M(f_{*} T)\leq Lip(f\vert_K)^p M(T), \,\,\,M(f_{*} \partial T)\leq Lip(f\vert_K)^{p-1} M(\partial T) .\]
In particular, this implies that $f_{*} T \in N_{p, f(K)}(V)$. Here $Lip(f\vert_K)$ denotes the Lipschitz constant of the restriction of the map $f$ to the compact subset $K$.
\end{lemma}

For any $T \in \mathcal{E}'(M)$, and $\alpha\in \mathcal{E}(M)$, set
\[ (\alpha \wedge T)(\phi)=T(\alpha \wedge \phi), \,\,\,\forall\,\phi\in \mathcal{E}(M).\]
The following simple lemma is quite useful in Section \ref{fundamental-result}.

\begin{lemma} \label{product-normal-chain} Let $K$ be a compact subset of $M$.
 If $T\in {\bf N}_K(M)$, then $\alpha \wedge T \in {\bf N}_K(M)$.

\end{lemma}
\begin{proof} We first claim that if $T\in \mathcal{E}'(M)$ has finite mass, then $\alpha\wedge T$ has finite mass. Without the loss of generality, we may assume that both $T$ and $\alpha$ are of homogeneous degrees. Suppose that $\text{supp}\, T\subset K$. Let $K_1$ be a compact subset of $M$ such that $K\subset \text{Inc}\, K_1$. Since $\text{supp}\, T\subset K$, we have that
\[ \vert \vert \alpha \wedge T\vert\vert =\text{sup}\,\{ T(\alpha \wedge \phi),\, \vert\vert\phi(x) \vert\vert \leq 1,\,\forall\, x\in K_1\}.\]
However, by \cite[1.8.1]{F69}, we have that
\[\vert\vert (\alpha\wedge \phi)(x)\vert\vert \leq \left(\begin{matrix} & p+q \\& p\end{matrix}\right) \cdot \vert \vert \alpha(x)\vert\vert \cdot \vert \vert\phi(x)\vert\vert,\, \forall \, x\in K_1,\] where $p$ and $q$ are the degrees of the form $\alpha$ and $\phi$ respectively.
It follows that \[\vert \vert \alpha \wedge T\vert\vert \leq \left(\begin{matrix} & p+q \\& p\end{matrix}\right) \cdot \vert\vert T \vert\vert \cdot \text{sup}_{x\in K_1}\,\vert \vert \alpha(x)\vert\vert.\]
This completes the proof of our claim. Next we observe that
\[(-1)^{\text{deg}\,\alpha} \partial (\alpha \wedge T)=\alpha \wedge \partial T - d\alpha \wedge T.\]
Therefore if both $T$ and $\partial T$ have finite mass, then $\partial (\alpha \wedge T)$ must have finite mass. This completes the proof of Lemma \ref{product-normal-chain}.


\end{proof}

By Lemma \ref{compatibility1}, on a coordinate neighborhood, our notion of normal current is equivalent to the one introduced in \cite[4.1.7]{F69}.
The next result asserts that on a differential manifold, any normal current is a finite sum of normal currents which each sits inside a coordinate neighborhood.

\begin{lemma} \label{partition} Let $V$ be an open subset of $M$, let $K\subset V$ be a compact set, and let $T \in N_{K}(V)$. Then there exist finite many coordinate neighborhoods $\{V_1,\cdots, V_r\}$ on $M$ which satisfy the following conditions.
\begin{itemize}

\item [1)] $K\subset \bigcup_i V_i\subset V$;

\item [2)] for any $1\leq i\leq r$, there exist a normal current $T_i\in N(V_i)$ supported inside $K$ such that
\[ T=\displaystyle\sum_{i=1}^r  T_i.\]

\end{itemize}

\end{lemma}

\begin{proof} We observe that by Lemma \ref{product-normal-chain}, for any smooth function $\rho$ supported inside a coordinate neighborhood $U$, $\rho\cdot T$ is a normal current supported inside $U\cap K$.  So Lemma \ref{partition} follows from an easy application of partition of unit.




\end{proof}

Now let $M$ and $N$ be two differential manifolds. We say that a map $f: M\rightarrow N$ is locally Lipschitzian if for any $x\in M$, there is a coordinate chart $(U,\psi)$ of $M$ near $x$, and a coordinate chart $(V,\phi)$ of $N$ near $f(x)$, such that $\phi\circ f\circ \psi^{-1}: \psi(U)\rightarrow \varphi(V)$ is a Lipschitzian map. Combining Lemma \ref{compatibility1},  Lemma \ref{Lipschitz-invariance-1}, and Lemma \ref{partition}, we see immediately the following result.

\begin{proposition}\label{Lipschitz-invariance-2}  Let $M$ and $N$ be two differential manifolds, let $f: M\rightarrow N$ be a locally Lipschitzian map, and let $K\subset M$ be a compact set. Then $\forall\, T\in N_K(M)$, we have that $f_{*}(T)\in N_{f(K)}(N)$.
\end{proposition}

Let $N$ be a differential manifold. We say that a $p$ dimensional compactly supported current $T$ of $N$ is a Lipschitz $p$-chain if there is a polyhedral $p$-chain
$Q$ of $\R^m$, and a locally Lipschitzian map $f: \R^m\rightarrow N$, such that $T=f_{*}(Q)$. We refer interested readers to \cite[ch.5]{GSM98} for a detailed treatment of Lipschitz chains in a manifold. For our purpose, we only need a simple property concerning Lipschitz chains. Roughly speaking, we show that given a Lipschitz $p$-chain $T$ of $N$,  if $p<\text{dim}\, N$, then the support of $T$ must be a proper subset of $N$.

Indeed, we will prove a more precise version of this property.  First we recall that for any $0\leq s<\infty$,
the $s$-dimensional Hausdorff measure $\mathcal{H}^s$ on $\R^m$ is well defined, c.f. \cite[2.10.2]{F69}. We will need the following simple result on
Hausdorff measure, and refer to \cite[2.10.2, 2.10.11]{F69} for detailed explanations.

\begin{lemma} \label{Hausdorff-measure-1} Let $f$ be a Lipschitzian map from $\R^n$ to $\R^m$, let $A$ be a Borel subset of $\R^n$, and let $0\leq s<\infty$.
Then we have that\[ \mathcal{H}^s(f(A))\leq (\text{Lip }f)^s\mathcal{H}^s(A).\]
\end{lemma}

Next, the Hausdorff dimension of a set $A\subset \R^m$ is defined to be \[\mathcal{H}_{\text{dim}}(A)=\text{inf}\{0\leq s<\infty,\, \mathcal{H}^s(A)=0\},\]
where $\mathcal{H}^s(A)$ is the $s$-dimensional Hausdorff measure of $A$. Clearly, the notion of Hausdorff dimension has the following extension to a manifold.

\begin{definition} \label{Hausdorff} Let $M$ be an $m$ dimensional differential manifold, and let $A\subset M$. For any $0\leq s<m$, we say that the $s$-dimensional Hausdorff measure of $A$ is zero if for any coordinate chart $(U,\varphi)$ of $M$, the $s$-dimensional Hausdorff measure of $\varphi(A\cap U)\subset \R^m$ is zero. We define the Hausdorff dimension of $A$ to be the infimum of the set of non-negative real numbers $s$ such that the $s$-dimensional Hausdorff measure of $A$ is zero.
\end{definition}

The following property of Lipschitz chains follows immediately from Lemma \ref{Hausdorff-measure-1}.

\begin{proposition}\label{Lipschitz-chain} Let $M$ be an $m$ dimensional differential manifold, and let $T$ be a Lipschitz $p$-chain, $0\leq p\leq m$. Then the Hausdorff dimension of the support of $T$ is at most $p$. In particular, if $p<m$, then the support of $T$ must be a proper subset of $M$, since the Hausdorff dimension of $M$ is clearly $m$.
\end{proposition}

Since any polyhedral chain in $\R^m$ must be a normal current, c.f. \cite[4.1]{F69}, by Proposition \ref{Lipschitz-invariance-2} we see that any Lipschitz chain in a differential manifold must be a normal current. This clearly implies the following result.

\begin{proposition}\label{polyhedral-chain-mfld}  Let $M$ be an $m$ dimensional differential manifold with a given polyhedral subdivision. Then
every polyhedral chain $T$ of the subdivision must be a normal current in $M$. \end{proposition}

\begin{definition}\label{cohomology-normal-currents} Consider the differential complex of normal currents

\begin{equation}\label{derham-complex-normal-currents}0\rightarrow N^0(M)\xrightarrow{d}N^1(M)\xrightarrow{d}\cdots \xrightarrow{d}N^{m}(M)\rightarrow 0.\end{equation}
We define the $i$-th cohomology of normal currents on $M$, denoted by $H^i_{\text{nor}}(M,\R)$, to be the $i$-th cohomology of the differential complex
(\ref{derham-complex-normal-currents}).

\end{definition}

The following Poincar\'e lemma is an immediate consequence of \cite[4.1.10]{F69} and Lemma \ref{compatibility1}.

\begin{lemma} Let $V$ be an $m$ dimensional  differential manifold which is homeomorphic to an open ball in $\R^m$. Then we have that

\[ H_{\text{nor}}^{i}(V,\R)= \begin{cases} 0,\,\,\,\,\text{if}\,\, 0\leq i<m\\
\R,\,\,\,\,\text{if}\,\, i=m.\end{cases}\]
\end{lemma}



Observe that if $M$ is an $m$ dimensional manifold covered by two open sets $U$ and $V$, then
 sequences of inclusions
\[ M\longleftarrow U \sqcup V \longleftarrow U\cap V\]
give rise to sequences of pushforward maps
\begin{equation*}
\begin{aligned}
     N^i(M)\xleftarrow{\text{sum}} N^i(U)\oplus N^i(V)
                      \xleftarrow{\text{signed pushforward}} & N^i(U\cap V)\\
    (-i_*(T), i_*(T)) \xleftarrow[\rule{2.39cm}{0cm}]{} & T
\end{aligned}
\end{equation*}
This provides us a Mayer-Vietoris sequence
\begin{equation}\label{Mayer-Vietois-sq-02}
0\longleftarrow N^i(M)\longleftarrow N^i(U)\oplus N^i(V) \longleftarrow N^i(U\cap V)\longleftarrow 0.
\end{equation}

\begin{proposition}\label{mayer-vietoris}
The Mayer-Vietoris sequence (\ref{Mayer-Vietois-sq-02}) of compactly supported currents is exact.

\end{proposition}
\begin{proof} The only thing needs to be checked is that the map
\begin{equation} \label{surjectivity-02} N^i(M)\xleftarrow{\text{sum}} N^i(U)\oplus N^i(V) \end{equation}
is surjective. Let $T\in N^i(M)$, and let $\{\rho_U,\rho_V\}$ be a partition of unit surbordinate to the open cover
$\{U,V\}$. Then by Lemma \ref{product-normal-chain}, we have that $\rho_U\cdot T\in N^i(U)$, and $\rho_V\cdot T\in N^i(V)$.
It follows that Map (\ref{surjectivity-02}) must be surjective.

\end{proof}

Then standard facts in homological algebra gives us the following long
exact sequence
\begin{equation}\label{exact-sq-2} \cdots \leftarrow H_{\text{nor}}^{i}(U\cup V)\leftarrow  H_{\text{nor}}^i(U)\oplus H_{\text{nor}}^i(V)\leftarrow H_c^{-\infty,i}(U\cap V) \leftarrow H_{\text{nor}}^{
i-1}(U\cup V)\leftarrow \cdots
\end{equation}

Note that the chain map given by the inclusion
\[ N(M)\hookrightarrow \mathcal{E}'(M)\]
induces a natural homomorphism of cohomologies
\begin{equation}\label{inlusion-map} H^*_{\text{nor}}(M)\rightarrow H_c^{\infty,*}(M).\end{equation}

In view of the exact sequences (\ref{exact-sq-1}) and (\ref{exact-sq-2}), the standard Mayer-Vietoris
argument as explained in \cite[Sec.5]{BT82} provides us the following result.

\begin{theorem}\label{De-Rahm-2} For any compact differential manifold $M$, the natural homomorphism
(\ref{inlusion-map}) is an isomorphism.
\end{theorem}

We will need the rescaled version of the Federer-Fleming deformation theorem. We conclude this section by giving the precise statement of the deformation theorem we need.

\begin{theorem}( \cite[4.2.9]{F69}, \cite[Sec. 5.1]{GSM98})\label{F-F-deformation} Let $\epsilon>0$, and let $T\in N_p(\R^N)$. Then there is a decomposition of $T$ as follows,
\[T=P+\partial R+S,\] where $P$ is a $p$ dimensional polyhedral chain in $\R^N$, and $R$ and $S$ are normal currents in $\R^N$ of dimension $p+1$ and $p$ respectively. Moreover, we have that
\[\begin{split} &\text{supp}\, P \cup \text{supp} R \subset \{ x, \,\text{disc}(x,\text{supp}\,T)\leq 2N\epsilon\},
\\&\text{supp}\, S\subset \{ x, \,\text{disc}(x,\text{supp}\,\partial T)\leq 2N\epsilon\}.\end{split}\]

\end{theorem}

\section{ $sl_2$ module structure on distributional De Rham complex}
\label{sl2-on-distribution-de-rham}

\medskip \noindent
\vskip 0.5cm

In this section, we discuss a canonical $sl_2$-module structure on the space of compactly supported currents on a symplectic manifold.

\begin{definition}\label{sl2-module}
Let $(M,\omega)$ be a $2n$ dimensional symplectic manifold,  $\pi=\omega^{-1}$ the canonical poisson bi-vector, and $\Pi^i: \mathcal{E}\rightarrow \mathcal{E}^i$  the projection operator. Define the Lefschetz map $L$, the dual Lefschetz map $\Lambda$, and the
degree counting map $H$ on $\mathcal{E}'$ as follows.
\begin{equation*} \begin{split}  & (L T)(\alpha)=T(\omega\wedge \alpha),\,\, \Lambda T(\alpha)=T(\iota_{\pi}\alpha),\,\,(H T)(\alpha)=T(-\sum_i(n-i)\Pi^i\alpha),
\,\\&\forall\,T\in \mathcal{E}',\,\,\forall\,\alpha \in\mathcal{E}.\end{split} \end{equation*}
In this context,  a compactly supported current $T$ of degree $i$ is said to be primitive if $L^{n-i+1}T=0$, where $0\leq i\leq n$.
\end{definition}

\begin{lemma}\label{commutator-on-currents-1} Consider the maps given in Definition \ref{sl2-module}.
We have the following commutator relations.

\[ [\Lambda ,L]=H, \,\,[H,\Lambda]=2\Lambda,\,\,\,[H,L]=-2L. \]
Therefore they define a natural $sl_2$ module structure on $\mathcal{E}'$.
\end{lemma}

\begin{proof} It is an immediate consequence of Definition \ref{sl2-module} and the usual commutator relations
on forms given in Equation \ref{sl2-module-on-forms}.
\end{proof}

Observe that although $\mathcal{E'}$ is an infinite dimensional $sl_2$-module, it has the property that $H$ has only finitely many eigenvalues.
The following result is a direct consequence of \cite[Corollary, 2.5, 2.6]{Yan96}.
\begin{proposition} \label{lefschetz-decomposition-for-compact-currents} Let $(M,\omega)$ be a $2n$ dimensional symplectic manifold. We have the following results.
\begin{itemize}
\item [1)] \[ L^k: \mathcal{E'}^{n-k}\rightarrow \mathcal{E'}^{n+k},\,\,T \mapsto \omega^{n-k}\wedge T\]
is an isomorphism for any $0\leq k\leq n$.
\item [2)]  $\forall\, 0\leq k \leq n$, $T \in \mathcal{E'}^k$ is primitive if and only if $\Lambda T=0$.

\item [3)]   Any  $T\in \mathcal{E'}^k$ admits a unique Lefschetz decomposition as follows.
\begin{equation} \label{lefeshetz-decomposition-for-currents} T=\displaystyle\sum_{r\geq \text{max}(\,\frac{k-n}{2}, 0)} \dfrac{L^r}{r!} T_{k-2r},\end{equation}
where $T_{k-2r}$ is a primitive compactly supported current in $\mathcal{E'}^{k-2r}$.

\end{itemize}

\end{proposition}
\begin{definition}\label{sym-hodge-star}
For any $0\leq i \leq 2n$, define the symplectic Hodge star operator on compactly supported currents by
\[(\star T)( \alpha)=T(\star \alpha),\,\,T\in\mathcal{E'}^{i}, \alpha\in \mathcal{E}^{i}.\]

And define the symplectic Hodge adjoint operator $d^{\Lambda}$ by
\[ (d^{\Lambda} T)=(d\Lambda -\Lambda d)T.\]
A compactly supported current $T$ is called (symplectic) Harmonic if and only if $dT=d^{\Lambda}T=0$.
\end{definition}






Now, using the commutator relations on forms given
 in the equation (\ref{commutator-on-forms}), it is straightforward to check that we have the following commutator relations
on currents.

\begin{lemma} \label{commutator-on-currents-2}
\[[d,L]=0,\,\,\,[d^{\Lambda},\Lambda]=0,\,\,\,[d,\Lambda]=d^{\Lambda},\,\,\,[d^{\Lambda},L]=d,\,\,\,[d,d^{\Lambda}]=0.\]
\end{lemma}

The following result is an easy consequence of \cite[Thm. 3.12]{We80} and the commutator relations given in Lemma \ref{commutator-on-currents-2}.

\begin{lemma}\label{non-commutative-polynomial} Consider the Lefschetz decomposition of a compactly supported current $T$ of degree $k$
as given in the equation (\ref{lefeshetz-decomposition-for-currents}). Then
\begin{itemize}\item [1)] there  are non-commutative polynomials $\Phi_{k,r}(L,\Lambda)$ such that
\[ T_{k-2r}=\Phi_{k,r}(L,\Lambda)T;\]

\item [2)] each $T_{k-2r}$ is $d$-closed and primitive if $T$ is Harmonic.
\end{itemize}
\end{lemma}

\section{A fundamental property of symplectic Harmonic forms}\label{fundamental-result}

We are now in a position to establish the main results of this paper.

\begin{theorem}\label{main-result-0} Let $(M,\omega)$ be a $2n$ dimensional compact symplectic manifold with a given polyhedral subdivision, and let $[\alpha]\in H^p(M,\R)$ be a primitive De Rham cohomology class of degree $p>0$. Suppose that in the space of compactly supported currents, $[\alpha]$ is cohomologous to a $2n-p$ dimensional polyhedral chain $Q$ of the subdivision. Then for any open neighborhood $W$ of $\text{supp} \,Q$, there exists a $2n-p$ dimensional compactly supported primitive current $T$ which is cohomologous to $[\alpha]$ in the space of compactly
supported currents, and which is supported inside the union of $W$ and the support of a $p-1$ dimensional Lipschitz chain.
In particular, $T$ can be chosen not to be supported on the entire manifold.
\end{theorem}

\begin{proof}
We first note that if $p=1$, then by dimension consideration, the set of interior points of $Q$ is a coisotropic submanifold.
So in this case, we can simply choose $T=Q$.

Therefore without the loss of generality we may assume that $2\leq p\leq n$. We divide the rest of the proof into two steps. We first lay out our approach in the framework of the cohomology of normal currents. We then apply
the Federer-Fleming deformation theorem to complete the proof.

{\bf Step 1: Algebraic setup.}

By assumption,
$\omega^{n-p+1}\wedge \alpha$ is cohomologous to zero. It follows that $\omega^{n-p+1}\wedge Q$, as a compactly supported current, is also cohomologous to zero.

By Lemma \ref{product-normal-chain} and Proposition \ref{polyhedral-chain-mfld}, $\omega^{n-p+1}\wedge Q$ is a normal current. It then follows from Theorem \ref{De-Rahm-2} that
there exists a $p-1$ dimensional normal current $\Gamma$ such that
\begin{equation}\label{being-a-coboundary} \omega^{n-p+1}\wedge Q= d\Gamma.\end{equation}

Since $d\Gamma=\pm \partial \Gamma$, we clearly have that \begin{equation}\label{support-0} \text{supp}\, \partial \Gamma= \text{supp}\,\left( \omega^{n-p+1}\wedge Q\right) \subset \text{supp}\, Q.\end{equation}
In order to apply the deformation theorem, we have the differential manifold $M$ smoothly embedded into some Euclidean space $\R^N$. Let
$\mathfrak{j}: M\rightarrow \R^N$ be the embedding, let $U$ and $V$ be two tubular neighborhoods of $\mathfrak{j}(M)$ in $\R^N$ such that $\overline{U}\subset V$ is a compact set, and let $\pi: V\rightarrow M$ be the retraction map which retracts $V$ onto $M$.

{\bf Step 2: Applying the deformation theorem.}

By Proposition \ref{Lipschitz-invariance-2}, $\mathfrak{j}_{*}(\Gamma)$ is a normal current in $U$. Applying Theorem \ref{F-F-deformation}, for any
 given constant $\epsilon>0$, we decompose
$\mathfrak{j}_{*}(\Gamma)$ as
\begin{equation}\label{decomposte-normal-current} \mathfrak{j}_{*}(\Gamma)=P+\partial R+ S,\end{equation}  where
 $P$ is a polyhedral chain of dimension $p-1$ in $\R^N$,  and $R$ and $S$ are normal currents in $\R^N$ of dimension
  $p$ and $p-1$ respectively; moreover,
 we have that
\begin{equation}\label{support-1}\begin{split}&\text{supp}\, P \cup \text{supp} R \subset \{ x, \,\text{disc}(x,\text{supp}\,\mathfrak{j}_*(\Gamma))\leq 2N\epsilon\},\\
&\text{supp}\, S\subset \{ x, \,\text{disc}(x,\text{supp}\,\partial \mathfrak{j}_*(\Gamma))\leq 2N\epsilon\}.\end{split}\end{equation}
Here for any subset $A\subset \R^N$, we denote by $\text{disc}(x,A)$ the distance function induced by the standard Euclidean norm on $\R^N$.

  Let $L$ be the Lipschitz constant of the composition map \[ \overline{U}\xrightarrow{\pi\vert_{\overline{U}}} \mathfrak{j}(M)\xhookrightarrow{\text{inclusion}} \R^N.\]
  Observe that $\text{supp}\, \partial\mathfrak{j}_*(\Gamma)\subset \mathfrak{j}_*(M)$. Thus the restriction of $\pi$ to $\text{supp}\,\partial\mathfrak{j}_*(\Gamma)$ is just the identity map. As a result, for any $x\in \text{supp}\,S$, we have that
  \begin{equation}\label{image-under-lipschitz-map}\text{disc}(\pi(x),\text{supp}\partial \mathfrak{j}_*( \Gamma)) \leq L \cdot \text{disc}(x,\text{supp}\partial \mathfrak{j}_*( \Gamma))
  \end{equation}

  Clearly, in view of (\ref{support-1}) and (\ref{image-under-lipschitz-map}), we may assume that $\epsilon$ is sufficiently small such that $P, R$ and $S$ are all supported inside $U$,  and such that
   \begin{equation}\label{support-2}  \text{supp} \, \pi_*(S)\subset W.\end{equation}

 Note that $\pi\circ \mathfrak{j}=\text{id}$. Therefore $\Gamma= \pi_{*}\circ\mathfrak{j}_{*}(\Gamma)=\pi_{*}(P)+\pi_{*}\partial (R)+
\pi_{*}(S)=\pi_{*}(P)+\partial\pi_{*} (R)+\pi_{*}(S)$. Since $d=\pm \partial$, it follows from (\ref{being-a-coboundary}) that
\begin{equation}\label{deformed} \omega^{n-p+1}\wedge Q=d\left(\pi_{*}(P)+\pi_{*}(S)\right).\end{equation}

Observe that $\pi_*(P)$ is a $p-1$ dimensional Lipschitz chain. By Proposition \ref{lefschetz-decomposition-for-compact-currents} there exists a $2n-p+1$ dimensional compact supported currents
$B$ such that
\begin{equation}\label{isomorphism-on-currents} \omega^{n-p+1}\wedge B=\pi_{*}(P)+\pi_{*}(S).\end{equation}

 Combining (\ref{deformed}) and (\ref{isomorphism-on-currents}), we have that
\[ \omega^{n-p+1}\wedge Q=\omega^{n-p+1}\wedge dB.\]

This shows that $\omega^{n-p+1}\wedge(Q-dB)=0$, i.e., $T:=Q-dB$ is a closed primitive current. Clearly, $Q$ is supported inside $W$. Since
$\text{supp}\,(d B)\subset \text{supp}\, B$, to finish the proof of Theorem \ref{main-result-0}, it suffices to show that
\begin{equation}\label{inclusion-relation} \text{supp}\, B\subset \text{supp}\left(\pi_{*}(P)+\pi_{*}(S)\right) \subset \text{supp}\pi_{*}(P)\cup W.\end{equation} To see this, applying Proposition \ref{lefschetz-decomposition-for-compact-currents}, we Lefschetz decompose
$B$ as follows.
\begin{equation}\label{decompose-B} B=\displaystyle\sum_{r\geq \text{max}(\,\frac{p-n-1}{2}, 0)} \dfrac{L^r}{r!} B_{p-2r-1},\end{equation}
where each $B_{p-2r-1}$ is a compactly supported primitive current of degree $p-2r-1$. Therefore we have that
\[ \pi_*(P)+\pi_*(S)=\displaystyle\sum_{r\geq \text{max}(\,\frac{p-n-1}{2}, 0)} \dfrac{L^{n-p+r-1}}{r!} B_{p-2r-1}.\]
Then by Lemma \ref{non-commutative-polynomial}, for each $r$, there exists a non-commutative polynomial $\Phi_r(L,\Lambda)$ such that
\begin{equation}\label{components-of-B} B_{p-2r-1}=\Phi_r(L,\Lambda)\left( \pi_*(P)+\pi_*(S)\right).\end{equation}
Thus we must have that
\[ \text{supp} B_{p-2r}\subset \text{supp}\left(\pi_{*}(P)+\pi_{*}(S)\right),\, \forall\, r.\]
This immediately implies (\ref{inclusion-relation}) and completes the proof.

\end{proof}

\begin{remark}\begin{itemize}
\item[a)] By Theorem \ref{duality}, every De Rham cohomology class is Poincar\'e dual to a polyhedral chain of the subdivision.

\item [b)] As an easy consequence of (\ref{components-of-B}), the current $B$ constructed in the proof of Theorem \ref{main-result-0} must have finite mass. However, it is not clear whether or not $B$ is a normal current in general.
We observe that $B$ is a $2n-p+1$ dimensional current supported inside $\text{supp\,}\pi_*(P+S)$. So if the Hausdorff dimension of $\text{supp\,}\pi_*(P+S)$ is strictly less than $2n-p+1$, then $B$ can never be a non-zero normal current. This is due to the basic fact that any flat currents, and thus any normal currents, can not be supported on lower dimensional sets. For more details, we refer to \cite[4.1.20]{F69}.
\end{itemize}
\end{remark}

We are going to prove a more precise version of Theorem \ref{main-result1}, which has Theorem \ref{main-result1} as an immediate consequence.

\begin{theorem}\label{main-result2}Let $(M,\omega)$ be a $2n$ dimensional compact  symplectic manifold with a given polyhedral subdivision. Then for any $1\leq p\leq n$, a primitive cohomology class of degree $p$ is represented by a symplectic Harmonic form $\alpha$ not supported on the entire manifold. More precisely, $\alpha$ is supported on a small open neighborhood of the union of the support of a $p-1$ dimensional Lipschitz chain, and the support of a $2n-p$ dimensional polyhedral chain associated to the given polyhedral subdivision.
\end{theorem}

\begin{proof}

It follows directly from Theorem \ref{main-result-0} and Theorem \ref{symplectic-smoothing-operator}.\mbox{\qedhere}
\end{proof}

\section{Harmonic representatives of Thom classes}\label{open-question}

\medskip \noindent
\vskip 0.5cm

In this section, we establish Theorem \ref{Guillemin's-question} and provide a satisfactory answer to the question asked by Victor Guillemin.

\begin{lemma}\label{cup-product} Let $(M,\omega)$ be a $2n$ dimensional compact symplectic manifold,  $N$ a compact oriented isotropic submanifold of $M$, and $[\tau_N]$ the Thom class of $N$. Then
$[\omega \wedge \tau_N]=0$.
\end{lemma}

\begin{proof}  Without the loss of generality, we may assume that $\text{codim }\, N\leq 2$.
Since $N$ is an isotropic submanifold of $M$,  $\omega\mid_N=0$.
Now for any close form $\beta$ of degree $ \text{dim}\,N -2$, we have

\[ \int_M \beta \wedge \omega \wedge \tau_N=\int_N \beta \wedge \omega =0.\]

It then follows from the Poincar\'{e} duality that $[\omega \wedge \tau_N]=0$.

\end{proof}





\begin{lemma}\label{Thom-class-of-general-submanifolds} Let $(M, \omega)$ be a $2n$-dimensional connected compact symplectic manifold, and let $\alpha$
be a symplectic Harmonic form of degree $k$. If $k$ is odd, or if $k=2p$ is even and
$L^{n-p}[\alpha]=0$, then we have the following Lefschetz decomposition of $\alpha$.
\begin{equation}\label{lefschetz-decompose-Thom-class} \alpha=\displaystyle \sum_{r\geq \text{max}(\frac{k-n}{2}, 0)} L^r \alpha_{k-2r}\end{equation}
where each $\alpha_{k-2r}$ is a closed primitive differential form of degree $k-2r>0$.
\end{lemma}

\begin{proof} Applying Lemma \ref{yan's-result-2}, we Lefschetz decompose $\alpha$ as follows.
\[ \alpha=\displaystyle \sum_{r\geq \text{max}(\frac{k-n}{2}, 0)} L^r \alpha_{k-2r},\]
where each $\alpha_{k-2r}$ is a closed primitive differential form of degree $k-2r$.
When $k$ is odd, Lemma \ref{Thom-class-of-general-submanifolds} follows automatically from dimension consideration.
Now assume that $k=2p$ is even and write
\begin{equation} \label{lefschetz-decompose-Thom-class-step2}
\alpha=\displaystyle\sum_{r=1}^p L^r \alpha_{2p-2r},\end{equation}
where $\alpha_{2p-2r}$ is a closed primitive differential form of degree $2p-2r$. It suffices to show that $\alpha_0=0$.
Indeed, since $\alpha_0$ is a closed form of degree zero, $\alpha_0$ must be a constant.

Apply $L^{n-p}$ to the both sides of Equation \ref{lefschetz-decompose-Thom-class-step2}. Since each differential form $\alpha_{2p-2r}$ is primitive, we get that
\[ L^{n-p}\alpha=\displaystyle\sum_{r=1}^p L^{n-p+r} \alpha_{2p-2r}=  \alpha_0\omega^n \]
must be $d$-exact. It follows that $\alpha_0=0$.
\end{proof}

The following result is an immediate consequence of Theorem \ref{main-result2} and Lemma \ref{Thom-class-of-general-submanifolds}.
\begin{theorem}\label{main-result4} Suppose that $(M,\omega)$ is a $2n$ dimensional connected compact symplectic manifold, and that $[\alpha] \in H^k(M,\R)$ is a De Rham cohomology class of degree $k$ which admits a symplectic Harmonic representative. Then $[\alpha]$ must admit a symplectic Harmonic representative not supported on the entire manifold, provided either one of the following two conditions are satisfied.
\begin{itemize}
\item[1)] $k$ is odd;
\item [2)] $k=2p$ is even, and $[\omega]^{n-p}\wedge [\alpha]=0$.
\end{itemize}

\end{theorem}

Now we are ready to prove Theorem \ref{Guillemin's-question}.

\begin{proof} We first point out that if $\text{codim}\, N=2p$ is even, and if $\omega^{n-p}\wedge[\tau_N]\neq 0$, then by
\cite[Thm.1]{Ba06} any symplectic Harmonic representative of $[\tau_N]$ is nowhere vanishing on $M$. For a treatment available in the literature, we refer to \cite[Lemma 4.1]{TY09}.

 The other statements in Theorem \ref{Guillemin's-question} follows directly from Theorem \ref{main-result4} and Lemma \ref{cup-product}.
\end{proof}

\begin{remark}
In view of Mathieu's theorem \cite{Ma95}, the assumption that the Thom class of the submanifold admits a symplectic Harmonic representative is necessary. For instance, it is easy to use \cite[Prop. 2.3]{L04} to show
that there are examples of compact oriented isotropic submanifolds whose
Thom classes do not have any symplectic Harmonic representatives.

\end{remark}

\appendix
\section{ A note on primitive cohomology}

A $2n$ dimensional symplectic manifold $M$ is said to satisfy the Hard Lefschetz property if and only if for any $0\leq k\leq n$, the Lefschetz map
  \begin{equation}\label{Lefschetz-map} L^k: H^{n-k}(M)\rightarrow H^{n+k}(M)\, \,\,[\alpha]\mapsto [\omega^k\wedge \alpha]
  \end{equation} is an isomorphism. In this appendix, we show that for symplectic manifolds with the Hard Lefschetz property,  the definition of primitive cohomology introduced in \cite{TY09} is equivalent to the usual
  definition. We first recall the following definition of primitive cohomology group given in \cite{TY09}.

\begin{definition}\label{primitive-cohomology-yau} Let $(M,\omega)$ be a $2n$ dimensional symplectic manifold, and let $P^{'r}(M)$ be the space of primitive $r$-forms which are closed under the symplectic adjoint operator $d^{\Lambda}$. Define \[PH^r_d (M)= \dfrac{ \text{ker }\, d\cap P^{'r}(M)}{ d  P^{'r-1}(M)},\,\,\,1\leq r\leq n.\]\end{definition}

 For the reader's convenience, we also include here the usual notion of primitive cohomology on a symplectic manifold. For a projective K\"ahler manifold, it is exactly what is used by algebraic geometers, cf. \cite[p. 4]{V07}.

\begin{definition}\label{primitive1} Let $(M,\omega)$ be a $2n$ dimensional symplectic manifold. For any $0\leq r \leq n$, the $r$-th primitive cohomology group, $PH^{r}(M)$, is defined by
 \[   PH^r (M)= \text{ker}(L^{n-r+1}: H^r(M)\rightarrow H^{2n-r+2}(M)) .\]
\end{definition}

To show that Definition \ref{primitive-cohomology-yau} is equivalent to Definition \ref{primitive1}, we need the following symplectic $dd^{\Lambda}$-lemma, which was independently established by Merkulov \cite{Mer98} and Guillemin \cite{Gui01}
\begin{theorem}\label{symplectic-ddelta-GM}  Suppose that $M$ is a compact symplectic manifold with the Hard Lefschetz property. Then we have
\[ \text{ ker} d\cap \text{ im} d^{\Lambda} =\text{im}d \cap \text{ker}d^{\Lambda} =\text{im}dd^{\Lambda}.\]
\end{theorem}

For any $0\leq k\leq n$, denote by $P^k(M)$ the space of primitive $k$-forms on $M$. The next simple result is an easy extension of the symplectic $dd^{\Lambda}$-lemma to primitive differential forms.

\begin{lemma}\label{primitive-ddelta}(\cite{TY10}) Let $\alpha \in P^k(M)$ with $0\leq k\leq n$. Suppose that there exists a $k$-form $\gamma$ such that $\alpha=dd^{\Lambda}\gamma$. Then there exists a primitive $k$-form $\beta$
such that $\alpha=dd^{\Lambda}\beta$.
\end{lemma}

\begin{proof} Lefschetz decompose $\gamma$ as follows
\[ \gamma= \beta_k+L\beta_{k-2}+L^2\beta_{k-4}+\cdots.\]
Here $\beta_{k-2i}$ is a primitive differential form of degree $k-2i$, $i=0,1,2,\cdots$. Since $[dd^{\Lambda},L]=0$, we get that
\begin{equation}\label{primitive-decom-and-ddelta} \alpha=dd^{\Lambda}\beta_k+Ldd^{\Lambda}\beta_{k-2}+L^2dd^{\Lambda}\beta_{k-4}+\cdots.\end{equation}
Since $[dd^{\Lambda}, \Lambda]=0$, the differential operator $dd^{\Lambda}$ maps primitive forms to primitive forms.
Therefore the right hand side of Equation \ref{primitive-decom-and-ddelta} is the Lefschetz decomposition of $\alpha$. Since $\alpha$ itself is a primitive form, it follows from the uniqueness of the Lefschetz decomposition that \[\alpha= dd^{\Lambda}\beta_k.\]
This completes the proof.

\end{proof}

\begin{proposition}\label{formality-1} Suppose that $M$ is a compact $2n$ dimensional symplectic manifold with the Hard Lefschetz property.
Then for any $0\leq r\leq n$,\[ PH^r_d(M) \cong PH^r(M).\]

\end{proposition}
\begin{proof} First we observe that on any symplectic manifold $M$, for any $0\leq r\leq n$, there is a natural homomorphism
\begin{equation}\label{natural-homomorphism} PH^r_d(M) \rightarrow PH^r(M),\,[\alpha]_{PH_d}\mapsto [\alpha]_{PH}.\end{equation}
Assume that $M$ is compact and satisfies the Hard Lefschetz property. We need to prove that this homomorphism is an isomorphism. When $r=0$, this is trivially true. We may assume that $r>0$.
Suppose that $\alpha$ is a closed form in $ P^{r'}(M)$ such that $[\alpha]_{PH}=0$. Since by definition $\alpha$ is $d^{\Lambda}$-closed,
 $\alpha$ is both $d$-exact and $d^{\Lambda}$-closed. It follows from Theorem \ref{symplectic-ddelta-GM} that $\alpha= dd^{\Lambda}\gamma$ for some $r$-form $\gamma$.
Since $\alpha$ is primitive, by Lemma \ref{primitive-ddelta}, we can assume that $\gamma$ is a primitive $k$-form.
Since $[d^{\Lambda},\Lambda]=0$ and since $\gamma$ is a primitive form, $d^{\Lambda}\gamma$ is also a primitive form.
It follows that $[\alpha]_{PH_d}=0$. This proves that the homomorphism (\ref{natural-homomorphism}) is injective.

If $[\alpha]_{PH}\in PH^r(M)$, then by definition $L^{n-r+1}\alpha$ represents a trivial cohomology class in $H^{2n-r+2}(M)$. Thus $L^{n-r+1}\alpha=d\beta_{2n-r+1}$ for a $(2n-r+1)$-form $\beta_{2n-r+1}$. It then follows from Lemma \ref{yan's-result} that $\beta_{2n-r+1}=L^{n-r+1}\eta$ for a $(r-1)$-form $\eta$.
Thus $L^{n-r+1}(\alpha-d\eta)=0$. Note that $\alpha-d\eta$ is both $d$-closed and primitive. So it must be $d^{\Lambda}$-closed as well.
Therefore $\alpha-d\eta$ represents a cohomology class in $PH^r_d(M)$ whose image under the homomorphism (\ref{natural-homomorphism}) is
$[\alpha]_{PH}$. This proves that the homomorphism (\ref{natural-homomorphism}) is also surjective.

\end{proof}






\medskip \noindent
\vskip 1.5cm

{\bf Acknowledgements.}
\medskip \noindent
\vskip 0.5cm
This paper has its origin in a joint project with Victor Guillemin. I am very grateful to him for inviting me to visit him at MIT, for suggesting to me the question of symplectic Harmonic Thom forms, and for generously sharing with me many of his insights on the subject. I would like to thank L. Tseng for very useful conversations when I was visiting MIT, and for his interest in this work. I would like to thank Reyer Sjamaar, Zuoqing Wang and Xiangdong Xie for a number of helpful discussions on this work, and for their friendship and moral support. Last but not least, I want to thank my entire family for being there for me. The most substantial part of this work is completed in the first seven months after my son Brian was born, while my daughter Catherine is growing into a two year old toddler. My wife Yishu has been a constant source of support in this endeavor.

\medskip

\noindent
Yi Lin \\
Department of Mathematical Sciences \\
Georgia Southern University\\
203 Georgia Ave., Statesboro, GA, 30460 \\
{\em E\--mail}: yilin@georgiasouthern.edu

\noindent
\noindent

\end{document}